\documentclass[3p,times]{elsarticle}

\usepackage{amssymb}

\usepackage{amsmath,amsthm}
\usepackage{amssymb,latexsym}
\usepackage{enumerate}

\numberwithin{equation}{section}

\newtheorem{theorem}{Theorem}[section]
\newtheorem{remark}[theorem]{Remark}
\newtheorem{proposition}[theorem]{Proposition}

\newtheorem{definition}[theorem]{Definition}

\newtheorem{The main theorem}[theorem]{The main theorem}

\theoremstyle{definition}

\allowdisplaybreaks

\begin{document}

\begin{frontmatter}




\title{Approximation of plurifinely plurisubharmonic functions}


\author[label1]{Nguyen Van Trao}
\address[label1]{Department of Mathematics, Hanoi National University of Education, Hanoi, Vietnam}  \ead{ngvtrao@yahoo.com}

\author[label3]{Hoang Viet}\address[label3]{Vietnam Education Publishing House, Hanoi, Vietnam}  \ead{viet.veph@gmail.com}

\author[label1]{Nguyen Xuan Hong \fnref{label1a}} 
 \fntext[label1a]{This work is finished during the third  author's post-doctoral fellowship of the Vietnam Institute for Advanced Study in Mathematics. He wishes to thank this institutions for their kind hospitality and support.}
\ead{xuanhongdhsp@yahoo.com}

\begin{abstract}
In this paper, we study   the approximation of negative plurifinely plurisubharmonic function defined on a plurifinely domain  by an increasing sequence of plurisubharmonic functions defined in   Euclidean domains.
\end{abstract}


\end{frontmatter}



\section{Introduction}
Approximation is   one of the most important tools in analysis.
Let  $\Omega$ be an Euclidean open set of  $\mathbb C^n$ and let $u$ be  a plurisubharmonic function on $\Omega$.
The problem of finding   characterizations of $u$ and $\Omega$ such that $u$ can be approximated uniformly on $\overline{\Omega}$  by a sequence of smooth plurisubharmonic functions defined  on Euclidean neighborhoods of  $\Omega$ is classical. 
The theorem of  Forn\ae ss and  Wiegerinck  \cite{FW89} asserts that it is always possible if $\Omega$  is bounded  domain with  $\mathcal C^1$-boundary and  $u$ is continuous on $\overline{\Omega}$. Recently,  Avelin,  Hed and   Persson \cite{AHP}  extended this result to domains with boundaries locally given by graphs of continuous functions.
Therefore, it makes sense not only to ask for which domains $\Omega$ such an approximation is  possible, but to ask for a characterization of those plurisubharmonic functions $u$ that can be monotonically approximated from outside. 
According to the results by \cite{Be06}, \cite{Be11}, \cite{CH08}, \cite{Hed10} and 
the third  author, this is possible if the  domain $\Omega$ has  the $\mathcal F$-approximation property and  $u$  belongs to the Cegrell's classes in $\Omega$.

The aim of this paper is to study approximation   of  $\mathcal F$-plurisubharmonic function.
More specifically,
let $u$ be a negative $\mathcal F$-plurisubharmonic function in $\mathcal F$-domain $\Omega$. We concern with sufficient conditions on $u$ and  $\Omega$ such that $u$ can be approximated  by an increasing  sequence of   plurisubharmonic functions defined  on Euclidean neighborhoods of  $\Omega$.  
It is not surprising that we need some kind of $\Omega$ and $u$ in analogy with the  set up to make the approximation possible. Namely, we prove the following. 

\begin{theorem} \label{the1}
Let  $\Omega$ be a bounded $\mathcal F$-hyperconvex domain and let  $\{\Omega_{j}\}$ be a decreasing sequence of  bounded hyperconvex domains such that $\Omega \subset \Omega_{j+1} \subset \Omega_j$, for all $j\geq 1$.  Assume that  there exists  $\rho \in\mathcal E_0(\Omega)$,  $\rho_j\in PSH^-(\Omega_j)$  with  $\rho_j\nearrow \rho$ a.e. in $\Omega$. Then, for  every $p>0$ and for  every $u\in\mathcal F_p(\Omega)$, there exists an increasing sequence of functions $u_j\in PSH^-(\Omega_j)$ such that $u_j\to  u$ a.e. in $\Omega$.
\end{theorem}

The paper is organized as follows. In section 2 we recall some notions of (plurifine) pluripotential theory. In Section 3, we give  the definition of bounded $\mathcal F$-hyperconvex domain $\Omega$ and the class $\mathcal E_0(\Omega)$  which is similar as the class introduced in \cite{Ce98} for the case  is a bounded hyperconvex domain. 
In Section 4, we introduce and investigate the class $\mathcal F_p(\Omega)$, $p>0$. 
Section 5 is devoted to prove  the theorem \ref{the1}.

\section{Preliminaries}
Some elements of pluripotential  theory (plurifine potential  theory) that will be used  throughout  the paper can be  found  in \cite{ACCH}-\cite{W12}. 
Let $\Omega$ be  a  Euclidean open set of $\mathbb C^n$. We  denote  by $PSH^-(\Omega)$   the  family  of   negative plurisubharmonic  functions in $\Omega$.
The plurifine topology $\mathcal F$ on   $\Omega$  is the smallest topology that makes all plurisubharmonic functions on $\Omega$ continuous. 
Notions pertaining to the plurifine topology are indicated with the prefix $\mathcal F$ to distinguish them from notions pertaining to the Euclidean topology on $\mathbb C^n$.
For a set $A\subset \mathbb  C^n$ we write $\overline{A}$ for the closure of $A$ in the one point compactification
of $\mathbb C^n$, $\overline{A}^{\mathcal F}$ for the $\mathcal F$-closure of $A$ and $\partial _{\mathcal F}A$ for the $\mathcal F$-boundary of $A$.
The set of all negative  $\mathcal F\text{-}$plurisubharmonic functions defined in $\mathcal F\text{-}$open set $\Omega$ is denoted by $\mathcal F\text{-}  PSH^-(\Omega)$. 

\begin{proposition}\label{Prrrr01010111}
Let $\Omega$ be an $\mathcal F$-open set in $\mathbb C^n$ and let $u\in \mathcal F\text{-}PSH^-(\Omega)$.
Assume that $\chi: \mathbb R^- \to \mathbb R^-$ is increasing  convex function. Then, $\chi\circ u \in \mathcal F\text{-}PSH^-(\Omega)$.
\end{proposition}

\begin{proof}
By Theorem 2.17 in \cite{KW14}, there exists a $\mathcal{ F}$-closed, pluripolar set $E\subset \Omega$ such that for every $z\in \Omega \backslash E$, there is an $\mathcal F$-open set $O_z \subset \Omega$ and a decreasing sequence of plurisubharmonic functions $\{\varphi_j\}$ defined in Euclidean open  neighborhoods of $O_z$ such that $\varphi_j \searrow u$ on $O_z$.
Since $\chi \circ \varphi$ is plurisubharmonic functions in Euclidean open neighborhoods of $O_z$ and $\chi \circ \varphi   \searrow \chi \circ u $ on $O_z$, by Theorem 3.9 in \cite{KFW11} we have $\chi \circ u \in \mathcal F\text{-}PSH^-(O_z)$. Therefore,  $\chi \circ u \in \mathcal F\text{-}PSH^-(\Omega \backslash E)$. Moreover, since $\chi\circ u$ is $\mathcal F$-continuous on $\Omega$. by Theorem 3.7 in \cite{KFW11}  it implies that 
 $\chi \circ u \in \mathcal F\text{-}PSH^-(\Omega )$.
The proof is complete.
\end{proof}

\begin{proposition}\label{Prrrr1}
Let $\Omega$ be an $\mathcal F$-open set in $\mathbb C^n$ and let $\varphi$ be strictly plurisubharmonic function in $\mathbb C^n$. Assume that $u,v\in \mathcal F\text{-}PSH^-(\Omega)$  such that 
$$\int_{\Omega \cap \{-\infty< u<v\}} (dd^c \varphi)^n =0.$$
Then, $u\geq v$ on $\Omega$.
\end{proposition}

\begin{proof}
Let $z\in \Omega \cap \{u>-\infty\}$ and $\lambda >0$ with $u(z)> - \lambda$. Choose $r>0$ and $\psi \in PSH^-(\mathbb B(z,r)) $ such that   $\psi( z)>-\frac{1}{2}$ and  $\mathbb B(z,r) \cap \{\psi >-1\} \subset \Omega$. Put 
$$
f:= 
\begin{cases}
\max( - 4\lambda , u+ 4\lambda \psi ) & \text{ in } \Omega
\\ - 4 \lambda & \text{ in }\mathbb B(z,r) \backslash \Omega
\end{cases}
$$ 
and 
$$
g:= 
\begin{cases}
\max( - 4\lambda , v+ 4\lambda \psi ) & \text{ in } \Omega
\\ - 4 \lambda & \text{ in }\mathbb B(z,r)  \backslash \Omega.
\end{cases}
$$ 
By Proposition 2.3 in \cite{KS14} and Proposition 2.14 in \cite{KFW11} it follows that  $f,g  \in PSH^-(\mathbb B(z,r)) $. From  the hypotheses  we have 
\begin{align*}
0\leq \int_{\mathbb B(z,r) \cap \{-\infty< f<g\}}  (dd^c \varphi)^n
\leq \int_{\Omega\cap \{-\infty< u<v\}}  (dd^c \varphi)^n
\leq 0.
\end{align*}
This implies that  $f \geq g$ in $\mathbb B(z,r)$. Hence, $u(z) \geq v(z)$, and therefore, $u \geq v$ in $\Omega \cap \{u>-\infty\}$.
Since $u,v$ are $\mathcal F$-continuous, by Theorem 3.7 in \cite{KFW11} we obtain that $u \geq v$ in $\Omega $.
The proof is complete.
\end{proof}

\begin{definition}{\rm 
Let $\Omega$ be an $\mathcal F$-open set in $\mathbb C^n$ and let  $QB(\Omega)$ be the trace of $QB(\mathbb C^n)$ on $\Omega$, where $QB(\mathbb C^n)$ is denoted  the measurable space on $\mathbb C^n$ generated by the Borel sets and the pluripolar subsets of  $\mathbb C^n$.
Assume that 
$u_1, \ldots , u_n \in \mathcal F\text{-} PSH(\Omega)$ be finite. 
Using the quasi-Lindel\"of property of plurifine topology and Theorem 2.17 in \cite{KW14}, there exist a pluripolar set $E\subset \Omega$, a sequence of $\mathcal F$-open subsets  $\{O_k\}$   and the plurisubharmonic functions  $f_{j,k}, g_{j,k}$ defined in Euclidean neighborhoods of $\overline O_k$  such that    $\Omega=E \cup \bigcup_{k=1}^\infty O_k $ and   $u_j=f_{j,k} -g_{j,k}$ on $O_k$. We define $O_0:= \emptyset$ and   
\begin{equation}\label{eqqqqqqqq}
\int_A dd^c u_1 \wedge \ldots \wedge dd^c u_n  
:= \sum_{j=1}^\infty \int_{A\cap ( O_j \backslash \bigcup_{k=0}^{j-1} O_k ) } dd^c ( f_{1,k} -g_{1,k}) \wedge \ldots  \wedge dd^c ( f_{n,k} -g_{n,k}), \  
A\in QB(\Omega). 
\end{equation} 
By  Theorem 3.6 in \cite{KW14}, the measure  defined by \eqref{eqqqqqqqq} is  independent on $E$,   $\{O_k\}$, $\{f_{j,k}\}$ and $\{g_{j,k}\}$.
This measure  is called the complex  Monge-Amp\`ere  measure.
}\end{definition}

\begin{proposition}
Let $\Omega$ be an $\mathcal F$-open set in $\mathbb C^n$  and let $u_1, \ldots , u_n \in \mathcal F\text{-} PSH(\Omega)$ be finite. Then,  $dd^c u_1 \wedge \ldots \wedge dd^c u_n  $ is non-negative measure  on $QB(\Omega)$.
\end{proposition}

\begin{proof}
The statement follows from \cite{BT87}, Theorem 2.17 in \cite{KW14} and Lemma 4.1 in \cite{KW14}. 
\end{proof}

\begin{proposition} \label{pr051}
Let $\Omega$ be an $\mathcal F\text{-}$open set in $\mathbb C^n$ and let $\mu$ be non-negative measure  on $QB(\Omega)$. Assume that 
$u,v\in \mathcal F\text{-}  PSH(\Omega)$ are finite such that $(dd^c u)^n \geq \mu$ and $(dd^c v)^n \geq \mu$ in $\Omega$. Then $(dd^c \max(u,v))^n \geq \mu $ in $\Omega$.
\end{proposition}

\begin{proof}
Put $v_j:=\max(u,v-\frac{1}{j})$, where $j\in \mathbb N^*$. By Theorem 4.8 in \cite{KW14} we have 
$$(dd^c v_j)^n  \geq 1_{\{ u\geq v\} } (dd^c u)^n + 1_{\{ u<v-\frac{1}{j}\}} (dd^c v)^n \geq 1_{\{ u\geq v\} \cup \{ u<v-\frac{1}{j}\}} \mu. 
$$
Since $v_j \nearrow \max(u,v)$ on $\Omega$, by Theorem 4.5 in \cite{KS14} we obtain 
$(dd^c \max(u,v))^n \geq \mu $ in $\Omega$.
The proof is complete. 
\end{proof}

\begin{proposition} \label{pr052}
Let $\Omega$ be an  $\mathcal F$-open set in $\mathbb C^n$  and let $u\in \mathcal F\text{-}PSH^-(\Omega)$ be finite. Assume that  $\{u_j\}$  is a monotone sequence of negative, finite, $\mathcal F$-plurisubharmonic functions such that $u_j\to u$ a.e. on $\Omega$. Then 
$$\int_\Omega f (dd^c u)^n 
\leq \liminf_{j\to+\infty} \int_\Omega f (dd^c u_j)^n,$$
for every non-negative, bounded, $\mathcal F$-continuous function $f$ on $\Omega$.
\end{proposition} 

\begin{proof}
From Theorem 3.9 in \cite{KFW11}, there exists a $\mathcal F$-closed, pluripolar set $E\subset \Omega$ such that $u_j \to u$ on $\Omega \backslash E$.
By Theorem 4.5 in \cite{KS14} we have  the sequence of measures $(dd^c u_j)^n$ converges $\mathcal F$-locally vaguely to $(dd^cu)^n$ on $\Omega \backslash E$.
Using the quasi-Lindel\"of property of plurifine topology,  there exist a pluripolar set $F\subset \Omega \backslash E$, a sequence of $\mathcal F$-open subsets  $\{O_k\}$   and non-negative $\mathcal F$-continuous  functions  $\chi_{k}$ in $\mathbb C^n$ with compact support on $O_k$ such that   $\Omega \backslash E=F \cup \bigcup_{k=1}^\infty O_k $, $0\leq \chi_k \leq 1$, $\sum_{k=1}^\infty \chi_k=1$ on $\Omega \backslash (E \cup F)$ and 
$$\int_{O_k} f \chi_k (dd^c u)^n= \lim_{j\to+\infty} \int_{O_k}  f\chi_k   (dd^c u_j)^n, \text{ for all } k\geq 1.$$
It follows that 
\begin{align*}
\int_\Omega f  (dd^c u)^n 
& =\int_{\bigcup_{k=1}^\infty O_k} f  (dd^c u)^n 
=\sup_{l\geq 1} \sum_{k=1}^l  \int_{O_k} f \chi_k (dd^c u)^n
\\&  =\sup_{l\geq 1} \lim_{j\to+\infty}  \int_\Omega f \left( \sum_{k=1}^l \chi_k   \right) (dd^c u_j)^n
\leq \liminf_{j\to+\infty} \int_\Omega f (dd^c u_j)^n.
\end{align*}
The proof is complete.
\end{proof}

\section{The class $\mathcal E_0(\Omega)$}

\begin{definition} \label{def1}
{\rm Let $\Omega$ be bounded $\mathcal F$-domain $\Omega$ in $\mathbb C^n$. Then,  $\Omega$   is  called  $\mathcal F$-hyperconvex if there exist a negative bounded plurisubharmonic function $\gamma_\Omega$ defined in a bounded hyperconvex domain $\Omega'$ such that $\Omega =\Omega' \cap \{\gamma_\Omega >-1\}$ and $-\gamma_\Omega $ is $\mathcal F$-plurisubharmonic  in $\Omega$.

We say that  a bounded negative  $\mathcal F$-plurisubharmonic function $u$ defined on bounded $\mathcal F$-hyperconvex domain $\Omega$ belongs to $\mathcal E_0(\Omega)$ if  $\int_\Omega (dd^c u)^n<+\infty$ and satisfy for every $\varepsilon>0$, there exists $\delta>0$ such that  $\overline{\Omega\cap  \{u<- \varepsilon\} } \subset  \Omega' \cap\{\gamma_\Omega >-1+\delta\}$. 
} \end{definition}

\begin{remark}{\rm 
If $\Omega$ is bounded hyperconvex domain then it is $\mathcal F$-hyperconvex. Moreover, there exists a bounded $\mathcal F$-hyperconvex domain that has no Euclidean interior point.

}\end{remark}

\begin{proposition} \label{pro1a}
Let $\Omega$ be a bounded $\mathcal F$-hyperconvex domain in $\mathbb C^n$. 
Then    $\mathcal E_0(\Omega) \neq \emptyset$. 
\end{proposition}

\begin{proof} 
Let $\Omega'$ be a bounded hyperconvex domain in $\mathbb C^n$ and let  $\gamma_\Omega \in PSH^-(\Omega') \cap L^\infty( \Omega')$ such that   $\Omega= \Omega' \cap \{\gamma_\Omega >-1\} $ and $-\gamma_\Omega  \in \mathcal F\text{-}PSH(\Omega)$. 
Let $\psi \in \mathcal E_0(\Omega')\cap \mathcal C(\Omega')$ such that $-1\leq  \psi<0$ in $\Omega'$.  
Choose $\varepsilon_0>0$ such that   $$G:=\{\psi<-2\varepsilon_0\} \cap \{\gamma_\Omega  >-1+2\varepsilon_0\} \neq \emptyset.$$ 
We define 
$$\rho:= \sup \{ \varphi \in \mathcal F\text{-}PSH^-(\Omega): \varphi \leq \max(-1 -\gamma_\Omega , \psi) \text{ on } G \}.$$
Since $\max(-1 -\gamma_\Omega , \psi) \in \mathcal F\text{-}PSH^-(\Omega)$ and  $G$ is $\mathcal F$-open set, we have $\rho \in \mathcal F\text{-}PSH^-(\Omega)$. 
Let $\varepsilon>0$. Choose $\delta \in (0,\varepsilon)$.
Because
$$-1\leq \max(-1 - \gamma_\Omega , \psi) \leq \rho <0 \text{ in } \Omega$$ 
and $\gamma_\Omega $ is upper semi-continuous on $\Omega'$, it follows that   
\begin{align*}
\overline{\{ \rho <-\varepsilon\}} 
&\subset \overline{\{\psi <-\varepsilon\} \cap \{\gamma_\Omega >-1+\varepsilon\}} 
\\ & \subset \{\psi \leq - \varepsilon\} \cap \{\gamma_\Omega  \geq -1+ \varepsilon\}
 \subset  \Omega ' \cap \{\gamma_\Omega  > -1+\delta\}. 
\end{align*} 
It remains to prove that $\int_\Omega (dd^c \rho)^n<+\infty$.
Put
$$u := 
\begin{cases}
\max(- \frac{1}{\varepsilon_0} , \rho  +  \frac{1}{\varepsilon_0} \gamma_\Omega ) & \text{ in } \Omega;
\\ - \frac{1}{\varepsilon_0}  & \text{ in }   \Omega' \backslash \Omega.
\end{cases}
$$
From Proposition 2.3 in \cite{KS14} and Proposition 2.14 in \cite{KFW11} we get $u \in PSH( \Omega')$. By Proposition 3.2 in \cite{KS14} we have $\rho$ is $\mathcal F$-maximal in $\{\psi>-2 \varepsilon _0\} \cup \{-1<\gamma_\Omega  <-1+ 2 \varepsilon _0 \}$. Moreover, since $\rho= u-\frac{1}{\varepsilon_0}\gamma_\Omega $ in $\{\gamma_\Omega >-1+\varepsilon_0\}$ and $\{\psi<-\varepsilon_0\} \Subset  \Omega'$, by Theorem 4.8 in \cite{KS14} it follows that 
\begin{align*}
\int_\Omega (dd^c \rho)^n 
&=\int_{\{\psi<-  \varepsilon_0\} \cup \{\gamma_\Omega  >-1+   \varepsilon_0\} } (dd^c \rho)^n 
\\& =\int_{\Omega' \cap (\{\psi<-  \varepsilon_0\} \cup \{\gamma_\Omega  >-1+   \varepsilon_0 \}) } (dd^c ( u-\frac{1}{\varepsilon_0}\gamma_\Omega ))^n 
\\ &  \leq \int_{\Omega' \cap \{\psi<-  \varepsilon_0 \}} (dd^c ( u+\frac{1}{\varepsilon_0}\gamma_\Omega ))^n  <+\infty
\end{align*}
(because $\Omega'\cap \{\psi<-\varepsilon_0\} \Subset \Omega'$). Therefore, $\rho \in \mathcal E_0(\Omega)$, and hence, $\mathcal E_0(\Omega) \neq \emptyset$. 
The proof is complete.
\end{proof}

\begin{proposition} \label{pro1b}
Let $\Omega$ be a bounded $\mathcal F$-hyperconvex domain in $\mathbb C^n$. Assume that   $u \in \mathcal E_0(\Omega)$ and $v\in \mathcal F \text{-}PSH(\Omega)$ such that $u\leq  v<0$ in $\Omega$.  Then, $v \in \mathcal E_0(\Omega)$ and 
$$ \int_\Omega  (-\rho) (dd^c v)^n \leq \int_\Omega (-\rho)  (dd^c u)^n,$$
for every $\rho \in \mathcal F \text{-}PSH^-(\Omega) \cap L^\infty(\Omega)$.  
Moreover, if $u=v$ in $\{u>-\varepsilon_0\}$ for some $\varepsilon_0>0$ then 
$$ \int_\Omega  (dd^c v)^n = \int_\Omega   (dd^c u)^n.$$
\end{proposition}

\begin{proof} 
Let $\Omega'$ be a bounded hyperconvex domain in $\mathbb C^n$ and let  $\gamma_\Omega \in PSH^-(\Omega') \cap L^\infty( \Omega')$ such that   $\Omega= \Omega' \cap \{\gamma_\Omega >-1\} $ and $-\gamma_\Omega  \in \mathcal F\text{-}PSH(\Omega)$. 
Fix $\varepsilon>0$. Choose $\delta>0$ such that $$\overline{\Omega \cap \{u<-\varepsilon\}} \subset \Omega' \cap \{\gamma_\Omega >-1+\delta\}.$$ 
Since $u\leq v<0$ in $\Omega$, we get 
$$\overline{\Omega \cap \{v<-\varepsilon\}} \subset  \overline{\Omega \cap \{u<-\varepsilon\}} \subset  \Omega ' \cap \{\gamma_\Omega >-1+\delta\}.$$ 
It remains to prove that 
$$ \int_\Omega  (-\rho) (dd^c v)^n \leq \int_\Omega (-\rho)  (dd^c u)^n,$$
for every $\rho \in \mathcal F \text{-}PSH^-(\Omega) \cap L^\infty(\Omega)$.  
We consider two cases follows.

{\em Case 1.} $u=v$ in $\Omega\cap \{u>-\varepsilon_0\}$ for some $\varepsilon_0>0$. Let $\psi \in \mathcal E_0(\Omega')\cap \mathcal C(\Omega')$. Choose $\delta_0>0$ such that 
$$ \Omega \cap  (\{ \psi >-2\delta_0\} \cup \{ \gamma_\Omega < -1+2 \delta_0 \})  \subset \Omega\cap \{u>-\varepsilon_0\} .$$ 
Without loss of generality we can assume that $-1\leq u \leq  v < 0$  and $-1\leq \rho \leq 0$  in $\Omega$.
Put  
$$f := 
\begin{cases}
\max(- \frac{1}{\delta_0}, u+  \frac{1}{\delta_0}  \gamma_\Omega ) & \text{ in } \Omega
\\ -  \frac{1}{\delta_0}& \text{ in } \Omega' \backslash \Omega
\end{cases}, \ 
g:= \begin{cases}
\max(- \frac{1}{\delta_0}, v+  \frac{1}{\delta_0}  \gamma_\Omega ) & \text{ in } \Omega
\\ -  \frac{1}{\delta_0}& \text{ in }   \Omega' \backslash \Omega
\end{cases}
$$ 
and 
$$ 
\varphi:= \begin{cases}
\max(- \frac{1}{\delta_0}, \rho +  \frac{1}{\delta_0} \gamma_\Omega ) & \text{ in } \Omega
\\ -  \frac{1}{\delta_0}& \text{ in } \Omega ' \backslash \Omega.
\end{cases} $$
From Proposition 2.3 in \cite{KS14} and Proposition 2.14 in \cite{KFW11} we get $f,g,\varphi \in PSH(\Omega ')$.  
By Theorem 4.8  in \cite{KW14} we have 
$$(dd^c u)^n =(dd^c v)^n \text{ in } \Omega \cap \{ \gamma_\Omega  <-1+2\delta_0\}.$$
Since $\rho=h- \frac{1}{\delta_0} \gamma_\Omega $, $u=f- \frac{1}{\delta_0}  \gamma_\Omega $, $v=g- \frac{1}{\delta_0} \gamma_\Omega $ in $\{ \gamma_\Omega  >-1+\delta_0\}$  and $f=g$ in $ \{\psi>-2\delta_0\} \cup \{ \gamma_\Omega <-1+2\delta_0\}$, 
by integration by parts yields
\begin{align*}
& \int_\Omega   \rho [(dd^c v)^n -(dd^c u)^n ] 
= \int_{\{ \gamma_\Omega  >-1+\delta_0 \}}  \rho [(dd^c v)^n -(dd^c u)^n ] 
\\& = \int_{\{ \gamma_\Omega  >-1+\delta_0\}}    ( \varphi - \frac{1}{\delta_0}  \gamma_\Omega  ) [(dd^c (g- \frac{1}{\delta_0} \gamma_\Omega ) )^n -(dd^c (f- \frac{1}{\delta_0}  \gamma_\Omega ) )^n ] 
\\ &= \int_{ \Omega '}  ( \varphi - \frac{1}{\delta_0}  \gamma_\Omega  ) [(dd^c (g- \frac{1}{\delta_0}  \gamma_\Omega ) )^n -(dd^c (f- \frac{1}{\delta_0} \gamma_\Omega ) )^n ] 
\\ & = \int_{ \Omega'}  ( \varphi - \frac{1}{\delta_0} \gamma_\Omega  ) dd^c (g-f) \wedge [\sum_{j=0}^{n-1} (dd^c  (f- \frac{1}{\delta_0}  \gamma_\Omega  ) )^j \wedge (dd^c (g- \frac{1}{\delta_0} \gamma_\Omega ) )^{n-j-1} ] 
\\ & = \int_{ \Omega'}(g-f) dd^c  ( \varphi - \frac{1}{\delta_0}  \gamma_\Omega )   \wedge [\sum_{j=0}^{n-1} (dd^c  (f- \frac{1}{\delta_0}  \gamma_\Omega  ) )^j \wedge (dd^c (g- \frac{1}{\delta_0} \gamma_\Omega  ) )^{n-j-1} ] 
\\& = \int_\Omega (v-u) dd^c \rho  \wedge [\sum_{j=0}^{n-1} (dd^c u)^j \wedge (dd^cv )^{n-j-1} ] 
\geq 0.
\end{align*} 
This follows that 
$$ \int_\Omega  (dd^c u)^n = \int_\Omega   (dd^c v)^n$$ 
and 
$$\int_\Omega (- \rho) (dd^c v)^n \leq \int_\Omega (-\rho)  (dd^c u)^n.$$

{\em Case 2.}  The general case.  Fix $\lambda \in (0,1)$ and define  
$$v_{j}=\max(u , \lambda  v- \frac{1}{j}), \text{ where } j\in\mathbb N^*.$$
Since $u=v_j$ in  $\{u> -\frac{1}{j} \}$, by the case 1 and Theorem 4.8 in \cite{KW14} we get 
\begin{align*}
\int_\Omega (- \rho) (dd^c u)^n 
& \geq \int_\Omega (-\rho)  (dd^c v_j)^n 
\\&  \geq   \int_{\{u<v_j\}} (-\rho)  (dd^c v_j)^n 
= \lambda  ^n \int_{\{u< \lambda  v -\frac{1}{j}\}} (-\rho)  (dd^c v)^n .
\end{align*}
It follows that 
\begin{align*}
\int_\Omega (- \rho) (dd^c u)^n 
& \geq \sup_{\lambda \in (0,1) } \left[ \lambda  ^n  \sup_{j\geq 1}   \int_{\{u< \lambda  v -\frac{1}{j}\}} (-\rho)  (dd^c v)^n \right]
\\& = \sup_{\lambda \in (0,1) } \left[ \lambda  ^n    \int_{\Omega} (-\rho)  (dd^c v)^n \right] 
=\int_{\Omega} (-\rho)  (dd^c v)^n.
\end{align*}
The proof is complete.
\end{proof}

\section{The class $\mathcal F_p(\Omega)$}

\begin{definition} \label{def1a}
{\rm 
Let $\Omega$ be a bounded $\mathcal F$-hyperconvex  domain in $\mathbb C^n$ and let $p>0$. Denote by  $\mathcal F_p (\Omega)$ is the family of negative $\mathcal F$-plurisubharmonic functions $u$ defined on $\Omega$ such that there exist a decreasing sequence $\{u_j\} \subset \mathcal E_0(\Omega)$ that converges pointwise to $u$ on $\Omega$ and 
$$\sup_{j\geq 1} \int_\Omega (1+(-u_j)^p)  (dd^c u_j)^n<+\infty.$$
} \end{definition}

\begin{remark}{\rm 
If $u\in \mathcal F_p(\Omega)$ then $u\in  \mathcal F_q(\Omega)$ for all $q\in (0,p)$.  
}\end{remark}

\begin{proposition} \label{pro2a}
Let $\Omega$ be a bounded $\mathcal F$-hyperconvex domain in $\mathbb C^n$ and let $p>0$.  Assume that  $u\in \mathcal F_p (\Omega)$   and  $\{u_j\} \subset \mathcal E_0(\Omega)$ such that $u_j \searrow u$ on $\Omega$ and 
$$\sup_{j\geq 1} \int_\Omega (1+(-u_j)^p)  (dd^c u_j)^n<+\infty.$$
Then, 
$$
\int_{\{ u>-\infty\}}   (dd^c u)^n 
= \sup_{j\geq1}\int_\Omega   (dd^c u_j)^n
.$$  
Moreover,   if $u$ is bounded then   
$$\int_\Omega (- v)  (dd^c u)^n 
= \sup_{j\geq1}\int_\Omega (-v) (dd^c u_j)^n, $$
for every  $v\in \mathcal F\text{-}PSH^-(\Omega) \cap L^\infty(\Omega)$.
\end{proposition}

\begin{proof}   
We consider two cases.

{\em Case 1.}   $u$ is bounded. 
First, we claim that if  $v \in \mathcal E_0(\Omega)$ then 
\begin{equation}\label{eq-----1} 
\int_\Omega (- v)  (dd^c u)^n 
= \sup_{j\geq1}\int_\Omega (-v) (dd^c u_j)^n.
\end{equation} 
Indeed,  without loss of generality we can assume that $-1\leq u\leq u_j <0$   in $\Omega$.  Let $\Omega'$ be a bounded hyperconvex domain in $\mathbb C^n$ and let  $\gamma_\Omega \in PSH^-(\Omega') \cap L^\infty( \Omega')$ such that   $\Omega= \Omega' \cap \{\gamma_\Omega >-1\} $ and $-\gamma_\Omega  \in \mathcal F\text{-}PSH(\Omega)$.  
Let $\{\delta_k\}$ be a decreasing sequence  of positive real numbers such that $\delta_k \searrow 0$ and $$\overline{\Omega\cap \{ v < - \frac{2}{k}\}} \subset   \Omega ' \cap \{\gamma_\Omega  >-1+ 2\delta_k\} \text{ for all } k\geq 1.$$
Define 
$$
\chi_k:=
\begin{cases}
 \max( \min( -k v-1, \frac{1}{\delta_k}( 1+\gamma_\Omega )  - 1 , 1),0) & \text{ in }\Omega;
\\    0 & \text{ in } \mathbb C^n \backslash \Omega.
 \end{cases} 
 $$
It is clear that $\chi_k$ is  $\mathcal F$-continuous  function with compact support on  $\Omega'$. Fix $k\geq 1$.
Put 
$$f := 
\begin{cases}
\max(-  \frac{1}{\delta_k}, u+  \frac{1}{\delta_k} \gamma_\Omega ) & \text{ in } \Omega
\\ -   \frac{1}{\delta_k}& \text{ in }  \Omega ' \backslash \Omega
\end{cases}$$ 
and 
$$
f_j := 
\begin{cases}
\max(-  \frac{1}{\delta_k}, u_j+   \frac{1}{\delta_k} \gamma_\Omega  ) & \text{ in } \Omega
\\ -  \frac{1}{\delta_k} & \text{ in }   \Omega' \backslash \Omega.
\end{cases}
$$ 
By Proposition 2.3 in \cite{KS14} and Proposition 2.14 in \cite{KFW11} it follows that  $f,f_j  \in PSH^-(\Omega ') \cap L^\infty( \Omega')$. Since $u=f - \frac{1}{\delta_k}\gamma_\Omega $, $u_j=f_j - \frac{1}{\delta_k} \gamma_\Omega $
in $\{\gamma_\Omega  >-1 + \delta_k \}$ and $\{\chi _k \neq 0\} \subset \{\gamma_\Omega >-1 +  \delta_k \}$, by \cite{BT87} we get 
\begin{align*}
\int_\Omega \chi_k  (-v) (dd^c u)^n 
&=\int_{\Omega'}  \chi_k  (-v) (dd^c (f -\frac{1}{ \delta_k} \gamma_\Omega ))^n  
\\& =\lim_{j\to+\infty} \int_{\Omega'}  \chi_k  (-v ) (dd^c (f_j -\frac{1}{ \delta_k} \gamma_\Omega  ))^n  
\\& = \lim_{j\to+\infty} \int_\Omega \chi_k  (- v) (dd^c u_j)^n .
\end{align*}
Moreover, since $\{\chi_k \neq 1\} \subset \{ v \geq  -\frac{2}{k}\} $, we get 
\begin{align*}
& \limsup_{j\to+\infty} \int_\Omega  (-v) (dd^c u_j)^n 
 \geq \int_\Omega \chi_k  (-v) (dd^c u)^n 
\\& \geq  \liminf_{j\to+\infty} \int_\Omega  (- v ) (dd^c u_j)^n  
- \limsup_{j\to+\infty} \int_\Omega  (1-\chi_k) (-v) (dd^c u_j)^n
\\& \geq  \liminf_{j\to+\infty} \int_\Omega  (-v) (dd^c u_j)^n  - \limsup_{j\to+\infty} \int_{\{ v \geq -\frac{2}{k}\}}   (- v ) (dd^c u_j)^n
\\& \geq  \liminf_{j\to+\infty} \int_\Omega  (-v) (dd^c u_j)^n  -  \frac{2}{k} \sup_{j\geq 1} \int_{\Omega}  (dd^c u_j)^n. 
\end{align*}
Let $k\nearrow +\infty$, by Proposition \ref{pro1b} we obtain that 
$$\int_\Omega  (-v) (dd^c u)^n
= \lim_{j\to+\infty} \int_\Omega  (-v) (dd^c u_j)^n
= \sup_{j\geq 1} \int_\Omega  (-v) (dd^c u_j)^n .$$
This proves the claim.  
Now, fix $\rho \in \mathcal E_0(\Omega)$ and define $v_k:=\max(v,k\rho)$, where $k\in \mathbb N^*$. By Proposition \ref{pro1b} it implies that $v_k \in \mathcal E_0(\Omega)$. Hence, by \eqref{eq-----1} and  Proposition \ref{pro1b}  we get 
\begin{equation} \label{eqho1}
\begin{split}
\int_\Omega  (-v) (dd^c u)^n 
&=\sup_{k\geq 1} \int_\Omega  (-v_k) (dd^c u)^n
\\&  = \sup_{k\geq 1} \left[ \sup_{j\geq 1} \int_\Omega  (-v_k) (dd^c u_j)^n \right] 
= \sup_{j\geq 1} \int_\Omega  (-v) (dd^c u_j)^n.
\end{split} 
\end{equation}

{\em Case 2.} The general case. Let $k\in\mathbb N^*$.  
Since $u_j \leq \max(u_j,-k)<0$ in $\Omega$, by Proposition \ref{pro1b} we have $\max(u_j,-k)\in \mathcal E_0(\Omega)$ and 
\begin{align*}
\sup_{j \geq 1}& \int_\Omega [1+(-\max(u_j,-k))^p] (dd^c \max(u_j,-k) )^n 
\\& \leq (1+k^p)  \sup_{j \geq 1}  \int_\Omega   (dd^c u_j)^n
<+\infty .
\end{align*}
Therefore,  $\max(u,-k) \in \mathcal F_p(\Omega)$. Hence, by \eqref{eqho1} and   Proposition \ref{pro1b}  we get 
\begin{align*}
\int_\Omega  (dd^c \max(u,-k))^n 
=\sup_{j\geq 1} \int_\Omega  (dd^c \max(u_j,-k))^n
=\sup_{j\geq 1}  \int_\Omega  (dd^c u_j )^n .
\end{align*}
Moreover, by Proposition \ref{Prrrr01010111} we have  $-(- u_m)^{\min(p,1)} \in \mathcal F\text{-} PSH^-(\Omega)$ for all $m\geq 1$. Hence, again by Proposition \ref{pro1b} it implies that 
\begin{align*}
\int_\Omega (- u)^{\min(p,1)} (dd^c \max(u,-k))^n 
&=\sup_{m\geq 1}  \int_\Omega   (- u_m)^{\min(p,1)}  (dd^c \max(u,-k))^n 
\\ & =\sup_{m\geq 1}  \left[ \sup_{j\geq 1} \int_\Omega (- u_m)^{\min(p,1)}  (dd^c \max(u_j ,-k))^n \right] 
\\ & \leq \sup_{m\geq 1}  \left[ \sup_{j\geq 1} \int_\Omega (- u_m)^{\min(p,1)}   (dd^c u_j )^n \right] 
\\& = \sup_{j\geq 1} \int_\Omega  (- u_j)^{\min(p,1)}  (dd^c u_j)^n 
\\& \leq  \sup_{j\geq 1} \int_\Omega  ( 1+ (- u_j)^{p})  (dd^c u_j)^n .
\end{align*}
It follows that
\begin{align*}
\int_{\{u\leq k\}}  (dd^c \max(u,-k))^n 
\leq \frac{1}{k^{\min(p,1)} }  
\sup_{j\geq 1} \int_\Omega  ( 1+ (- u_j)^{p})  (dd^c u_j)^n .
\end{align*}
Therefore, by Theorem 4.8 in \cite{KW14} we get 
\begin{align*}
\int_{\{ u>-\infty\}}   (dd^c u)^n 
& = \lim_{k\to+\infty} \int_{\{u>-k\}}  (dd^c \max(u,-k))^n 
\\& = \lim_{k\to+\infty} \int_\Omega  (dd^c \max(u,-k))^n 
= \sup_{j\geq 1}  \int_\Omega  (dd^c u_j )^n .
\end{align*}
The proof is complete. 
\end{proof}


\begin{proposition} \label{pro3}
Let $\Omega$ be a bounded $\mathcal F$-hyperconvex domain in $\mathbb C^n$ and let $p>0$.  Assume that  $u\in \mathcal F_p (\Omega)$    and   $v\in \mathcal F\text{-}PSH(\Omega)$ with $u\leq v<0$   then $v\in \mathcal F_{\min(p,1)}  (\Omega)$ and 
$$\int_{\{ v>-\infty\}}   (dd^c v)^n \leq  \int_{\{ u>-\infty\}}   (dd^c u)^n.$$
\end{proposition}

\begin{proof}
Let $\{u_j\}\subset \mathcal E_0(\Omega)$ such that $u_j \searrow u$ in $\Omega$ and 
$$\sup_{j\geq 1}  \int_\Omega (1+(-u_j)^p) (dd^c u_j)^n <+\infty.$$
Put $v_j:=\max(u_j,v)$. By Proposition \ref{pro1b} we have $v_j \in \mathcal E_0(\Omega)$.  Moreover, by  Proposition \ref{Prrrr01010111}
we have $-(-v_j)^{\min(p,1)} \in \mathcal F\text{-}PSH^-(\Omega)$. Hence,  again by Proposition \ref{pro1b} it implies that 
\begin{align*}
\sup_{j\geq 1} \int_\Omega [1+(-v_j)^{\min(p,1)}](dd^c v_j)^n 
&\leq \sup_{j\geq 1} \int_\Omega [1+(-v_j)^{\min(p,1)}](dd^c u_j)^n 
\\& \leq \sup_{j\geq 1}  \int_\Omega [2+(-u_j) ^p] (dd^c u_j)^n <+\infty.
\end{align*}
Since $v_j \searrow v$ in $\Omega$, it implies that 
$v\in \mathcal F_{\min(p,1)} (\Omega)$. Therefore, by Proposition \ref{pro1b} and Proposition \ref{pro2a} we obtain 
\begin{align*}
\int_{\{ v>-\infty\}}  (dd^c v)^n 
& =\sup_{j\geq 1} \int_\Omega   (dd^c v_j)^n 
\\& \leq \sup_{j\geq 1}  \int_\Omega  (dd^c u_j)^n
= \int_{\{ u>-\infty\}}   (dd^c u)^n.
\end{align*}
The proof is complete. 
\end{proof}

\begin{proposition} \label{pro4}
Let $\Omega$ be a bounded $\mathcal F$-hyperconvex domain in $\mathbb C^n$ and let $p>0$. Assume that  $u\in \mathcal F_{\min(p,1)} (\Omega)$ and $v\in \mathcal F\text{-}PSH^-(\Omega)$ such that   $(1+(-u)^p)  (dd^c u)^n \leq  (1+(-v)^p) (dd^c v)^n$ in $\Omega \cap \{u>-\infty\} \cap \{v>-\infty\}$.  
Then $u \geq v$ in $\Omega$. 
\end{proposition}

\begin{proof}
Let  $\varphi$ be smooth  strictly plurisubharmonic function in $\mathbb C^n$ such that $\Omega\subset \{\varphi <0\}$. Put $v_j:= \max(u, v+\frac{1}{j} \varphi)$ on $\Omega$, where $j\in \mathbb N^*$. 
First, we claim that 
$$(dd^c v_j)^n   \geq  (dd^c u)^n \text{ in }\Omega \cap \{u>-\infty\} \cap \{v>-\infty\} .$$
Indeed,  by the hypotheses it implies that 
\begin{align*}
(dd^c (v+\frac{1}{j} \varphi) )^n   
\geq   (dd^c v )^n   
\geq   \frac{1+(-u)^p}{1+(-v)^p } (dd^c u )^n    
\geq 1_{\{u\leq v+\frac{1}{j} \varphi\}} (dd^c u)^n 
\end{align*}
in $\Omega \cap \{u>-\infty\} \cap \{v>-\infty\}$. Hence, by Proposition \ref{pr051} we get 
$$ 1_{\{u\leq v +\frac{1}{j} \varphi \}}  (dd^c v_j )^n   \geq 1_{\{u\leq v +\frac{1}{j} \varphi \}} (dd^c u)^n \text{ in } \Omega \cap \{u>-\infty\} \cap \{v>-\infty\}.$$
Moreover, by Theorem 4.8  in \cite{KW14} we have 
$ (dd^c v_j )^n  = (dd^c u)^n$ in $\Omega \cap \{u> v +\frac{1}{j}\varphi \} \cap \{u>-\infty\} \cap \{v>-\infty\}$.
Therefore, 
$$(dd^c v_j )^n    \geq (dd^c u)^n  \text{ in } \Omega \cap \{u >-\infty\} \cap \{v> -\infty\} .$$
This proves the claim.
Since $v_j \geq u$ in $\Omega$, by Proposition \ref{pro3} we have 
\begin{align*}
\int_{\{ u>-\infty\}}  (dd^c u)^n \leq \int_{\{ v_j >-\infty\}}  (dd^c v_j )^n \leq \int_{\{ u>-\infty\}}  (dd^c u)^n<+\infty.
\end{align*}
It follows that $1_{\{ v_j >-\infty\}} (dd^c v_j)^n = 1_{\{ u>-\infty\}}  (dd^c u)^n$ in $\Omega$. Therefore, by Theorem 4.8 in \cite{KW14} we get 
\begin{align*}
\int_{\{-\infty< u<v_j\}}  (dd^c \varphi)^n 
&\leq j^n  \int_{\{-\infty< u<v_j\}} [ (dd^c (v+ \frac{1}{j} \varphi) )^n -   (dd^c v)^n ]
\\& \leq j^n  \int_{\{-\infty< u<v_j\}}  [(dd^c v_j )^n -   (dd^c u)^n] =0.
\end{align*}
Thus, 
$$
\int_{\{-\infty< u<v\}}  (dd^c \varphi)^n
= \sup _{j\geq 1} \int_{\{-\infty< u<v_j\}}  (dd^c \varphi)^n  
=0.
$$
From Proposition \ref{Prrrr1} we have 
$u\geq v$ on $\Omega$. 
The proof is complete.
\end{proof}

\section{Proof of theorem \ref{the1}}

\begin{proof} 
Let $\{\varphi_j\} \subset \mathcal E_0(\Omega)$ such that $\varphi_j \searrow u$ on $\Omega$ and 
$$\sup_{j\geq 1} \int_\Omega (1+(-\varphi_j)^p)  (dd^c \varphi_j)^n<+\infty.$$
By Proposition \ref{pr052} we have 
\begin{align*}
\int_{\{u>-\infty\}}  (1+(-u)^p)  (dd^c u)^n
& \leq \sup_{k\geq 1} \int_{\Omega \cap \{u>-\infty\}}  (1+(-\varphi_k )^p) (dd^c u)^n
\\ & \leq  \sup_{k \geq 1} \left[ \liminf_{j\to+\infty} \int_\Omega (1+(-\varphi_k)^p)  (dd^c \varphi_j)^n\right] 
\\ & \leq  \sup_{j \geq 1}   \int_\Omega (1+(-\varphi_j)^p)  (dd^c \varphi_j)^n 
<+\infty .
\end{align*}
Moreover, since the   measure $1_{\Omega\cap \{u>-\infty\}}  (1+(-u)^p) (dd^c u)^n$ vanishes on all pluripolar subsets of $\Omega_j$, by Theorem 4.10 in \cite{HH11} there exists $u_j \in \mathcal F_p(\Omega_j)$ such that 
$$ (1+(-u_j)^p) (dd^c u_j)^n=1_{ \Omega \cap \{u>-\infty\} }  (1 +(-u)^p) (dd^c u)^n \text{ in } \Omega_j. $$
By Theorem 4.8 in \cite{HH11}  we have  $u_j \geq u_{j+1}$ in $\Omega_{j+1}$. Moreover, since $u\in \mathcal F_{\min(p,1)}(\Omega)$,  by  Proposition \ref{pro4}  it implies that  $ u\geq   u_j$ in $\Omega$ for all $j\geq 1$. Let $v$ be the least $\mathcal F$-upper semi-continuous regularization of $\lim_{j\to+\infty} u_j$ on $\Omega$. By Theorem 3.9 in \cite{KFW11} we get $u_j \to v $ a.e in $\Omega$. 

We claim that $v\in \mathcal F_{\min(p,1)} (\Omega)$. Indeed, 
put $v_k:=\max(v,k \rho)$, where $k\in \mathbb N^*$. 
By Proposition \ref{pro1b} we have $v_k\in \mathcal E_0(\Omega)$. Since $\max(u_j, k\rho_j) \nearrow v_k$ a.e. in $\Omega$, by    Proposition  \ref{Prrrr01010111},   Proposition \ref{pr052} and Lemma 3.3 in \cite{ACCH}   we get
\begin{align*}
\int_\Omega [1+(-v_k)^{\min(p,1)}] (dd^c v_k)^n 
& \leq \liminf_{j\to+\infty} \int_\Omega [1+(-v_k )^ {\min(p,1)} ](dd^c \max(u_j, k\rho_j))^n 
\\& \leq \liminf_{j\to+\infty} \int_{\Omega_j}  [ 1+( -\max(u_j, k\rho_j))^{\min(p,1)}] (dd^c \max(u_j, k\rho_j))^n 
\\ & \leq \liminf_{j\to+\infty} \int_{\Omega_j} [1+(-u_j)^{\min(p,1)}] (dd^c u_j)^n 
\\ & \leq 2 \liminf_{j\to+\infty} \int_{\Omega_j} [1+(-u_j)^p] (dd^c u_j)^n 
\\& =  2 \int_{\Omega\cap  \{u>-\infty\}} (1+(-u)^p) (dd^c u)^n .
\end{align*}
Hence, 
\begin{align*}
\sup_{k\geq 1} \int_\Omega [1+(-v_k)^{\min(p,1)}] (dd^c v_k)^n 
\leq  2 \int_{\Omega\cap \{u>-\infty\}} (1+(-u)^p) (dd^c u)^n <+\infty.
\end{align*}
It follows that $v\in \mathcal F_{\min(p,1)}(\Omega)$. This proves the claim.
Since $v\geq u_j$  in $\Omega$, so 
$$  (1+(-v)^p) (dd^c u_j)^n \leq  (1 +(-u)^p) (dd^c u)^n    \text{ in }  \Omega \cap \{u>-\infty\}. $$ 
Moreover, since $u_j \nearrow v$ a.e. in $\Omega$, by Theorem 4.5 in \cite{KS14} we have 
$$ (1+(-v)^p) (dd^c v)^n \leq  (1 +(-u)^p) (dd^c u)^n \text{ in }  \Omega \cap \{u>-\infty\} \cap \{v>-\infty\} . $$ 
Hence, by Proposition 
\ref{pro4} it implies that $v \geq u $ in $\Omega$, and therefore, $u=v$ in $\Omega$. Thus,  $u_j\to u$ a.e. in $\Omega$. The proof is complete.
\end{proof}





\bibliographystyle{elsarticle-num}
\bibliography{<your-bib-database>}



\end{document}